\documentclass[12pt]{amsart}
\usepackage[margin=1in]{geometry}

\usepackage{standalone}
\usepackage{enumerate}

\usepackage{tikz}
\usepackage{tikz-cd}

\usepackage{amssymb}
\usepackage{mathtools}
\usepackage{thmtools,thm-restate}

\usepackage{parskip}

\usepackage{etoolbox}
\AtEndEnvironment{thm}{\vspace{-0.5\baselineskip}}
\AtEndEnvironment{prop}{\vspace{-0.5\baselineskip}}
\AtEndEnvironment{lem}{\vspace{-0.5\baselineskip}}
\AtEndEnvironment{cor}{\vspace{-0.5\baselineskip}}
\AtEndEnvironment{rem}{\vspace{-0.5\baselineskip}}

\theoremstyle{plain}
\newtheorem{thm}{Theorem}[section]
\newtheorem{prop}[thm]{Proposition}

\newtheorem{lem}[thm]{Lemma}

\theoremstyle{definition}
\newtheorem{defn}[thm]{Definition}

\newtheorem{rem}[thm]{Remark}
\newtheorem{Not}[thm]{Notation}
\newtheorem{Con}[thm]{Convention}

\newtheorem*{ack}{Acknowledgements}

\usepackage[all]{xy}

\usepackage{graphicx}
\graphicspath{{Images/}}
\usepackage{subcaption}

\usepackage{todonotes}

\usepackage[shortlabels]{enumitem}

\usepackage[british]{isodate}

\usepackage{bm}

\usepackage{hyperref}

\newcommand{\N}{\mathbb{N}}
\newcommand{\Z}{\mathbb{Z}}

\DeclareMathOperator{\Aut}{Aut}
\DeclareMathOperator{\Out}{Out}
\DeclareMathOperator{\Inn}{Inn}

\DeclareMathOperator{\Fix}{Fix}

\DeclareMathOperator{\Ad}{Ad}
\DeclareMathOperator{\MC}{Mc}

\newcommand{\GG}{\mathcal{G}}
\DeclareMathOperator{\IA}{IA}

\let\phi\varphi

\let\hat\widehat

\title{Centralisers of linear growth automorphisms of free groups}
\author{Naomi Andrew} 
\author{Armando Martino}
\address{Mathematical Sciences, Building 54, University of Southampton, Southampton, SO17 1BJ}
\email{naomi.maths@gmail.com}
\email{A.Martino@soton.ac.uk}
\address{Current Address (NA): Mathematical Institute, Andrew Wiles Building, Observatory Quarter, University of Oxford, Oxford OX2 6GG, United Kingdom}
\date{}

\begin{document}

\begin{abstract}
	In this note we investigate the centraliser of a linearly growing element of $\Out(F_n)$ (that is, a root of a Dehn twist automorphism), and show that it has a finite index subgroup mapping onto a direct product of certain ``equivariant McCool groups'' with kernel a finitely generated free abelian group. In particular, this allows us to show it is VF and hence finitely presented.
\end{abstract}
\maketitle
	
	
\section{Introduction}
	
Let $G$ be a group, and consider the centraliser $C_G(g)$ of an element $g$ of $G$. Understanding centralisers of elements is related to solving the conjugacy problem -- it serves as a kind of dual problem -- and for calculating the virtually cyclic dimension of the group \cite{HughesGuerchSanchezSaldana2023}. It also has implications for actions of $G$ on a CAT(0) space: if an element $g$ has finite order in the abelianisation of its centraliser $C(g)^\textrm{ab}$	then it cannot act loxodromically in any action on a CAT(0) space \cite{Bridson2010CentralisersCAT0}.

In this paper, we consider elements $\Phi$ of $\Out(F_n)$, the outer automorphism group of a free group of rank $n$. It is an open problem, in general, to prove that centralisers of free group outer automorphisms are finitely generated, even though some important cases are well understood: for instance when the outer automorphism is irreducible, its centraliser is virtually cyclic by \cite{Bestvina1997}.

By contrast, in the mapping class group of a surface, the centraliser of a \emph{pure} element -- where the reduction system is fixed up to isotopy and the restrictions to complementary surfaces are either fixed or pseudo-Anosov -- is Type VF by work of Ivanov \cite{Ivanov1992MCGSubgroups} and is known to be finitely generated in general, with an algorithm producing a generating set due to Rafi, Selinger and Yampolsky \cite{Rafi2020CentralisersMCG}. 

A group $G$ is of Type F if it has a finite Eilenberg--Maclane space, its $K(G,1)$. A group is said to be of Type VF if it has a finite index subgroup of Type F. Groups of Type VF are finitely generated and finitely presented; since groups of Type F are necessarily torsion free, being Type VF is the strongest homotopic finiteness property available for a group with finite order elements.

However we note that deducing information about a centraliser from a power -- from a pure mapping class, say -- is not straightforward. Our example in Section~\ref{sec:ex} shows that in general the centraliser of a group element need not be commensurable with the centraliser of some power of the element even if the element has infinite order. (In fact our example can easily be made to be geometric, where the elements are actual mapping classes.) This stands in contrast to the situation within the finite index subgroup $\IA_3(n)$ of $\Out(F_n)$, as investigated by Guerch \cite{Guerch2022}.

In \cite{CullerVogtmann1986CVn} Culler and Vogtmann construct a contractible simplicial complex on which $\Out(F_n)$ acts with finite stabilisers. Together with the existence of a torsion-free, finite index subgroup $\IA_3(n)$ of $\Out(F_n)$, this shows that $\Out(F_n)$ has Type VF. Later, Krstic and Vogtmann used the fixed point subsets of this complex for finite order elements to show that the same is true for centralisers of finite order elements of $\Out(F_n)$ \cite{KrsticVogtmann1993}. In \cite{CentralisersDehnTwists}, Rodenhausen and Wade consider Dehn twist automorphisms of free groups: these preserve certain cyclic splittings of $F_n$ in the same way a Dehn (multi)twist of a surface preserves the cyclic splitting of the fundamental group dual to the twisting curves. They prove that their centralisers are of type VF, using their ``efficient'' representatives as graph of groups automorphisms due to Cohen and Lustig \cite{Cohen1999}. Here we prove the same result for automorphisms of linear growth, which are exactly the roots of the Dehn twists: in some sense this is analogous to passing from the trivial element (centralised by all of $\Out(F_n)$) to a finite order element. 

\subsection*{Results}

Before stating our results more precisely we draw the reader's attention to our convention of writing maps on the \textit{right}. 

\begin{Con}
	\label{right}
Let $G$ be a group. 
\begin{itemize}
	\item We write lower case greek letters, $\phi$, for automorphisms of $G$ and upper case Greek letters, $\Phi$ for outer automorphisms of $G$. 
	\item We write our automorphisms on the right; $g \phi$ is the image of $g \in G$ under the automorphism $\phi$. 
	\item If $x \in G$, we write $\Ad(x)$ to denote the inner automorphism of $G$ so that
	\[
	g \Ad(x) = x^{-1} g x.
	\]  
	
	With this convention we get that $\Ad(x) \Ad(y) = \Ad(xy)$ and $\Ad(x) \phi = \phi \Ad( x \phi)$. 
	\item Accordingly we write $g^x:=x^{-1} g x$ 
\end{itemize}
\end{Con}
	
Our main theorem is the following: 	
	
\begin{restatable}{thm}{centralisers}
	\label{mainthm}
	Let $\Phi$ be a linearly growing element of $\Out(F_n)$. Then $C(\Phi)$ is of type VF. 
	
	More precisely,  $C(\Phi)$  admits a finite index subgroup, $C_0(\Phi)$, which fits into a short exact sequence: 
	\[1 \longrightarrow \Z^m \longrightarrow C_0(\Phi) \longrightarrow N \longrightarrow 1.\]
	
	where $N$ is a finite index subgroup of a finite product of equivariant McCool groups, $\prod_{u\in U}\MC_u$. The indexing set $U$ is a finite set and each \textit{equivariant McCool group}, $\MC_u $, (and hence their product) is of type VF.
\end{restatable}
	
Equivariant McCool groups are defined in Definition~\ref{def:mccool}. They are subgroups of $\Out(F_n)$, taking as input a collection of conjugacy classes of $F_n$ to preserve and a subgroup of $\Out(F_n)$ to centralise. McCool first studied the subgroups of $\Out(F_n)$ preserving a set of conjugacy classes; some examples arise from mapping class groups of punctured surfaces with boundary words corresponding to the given conjugacy classes. Bestvina, Feighn and Handel \cite{BestvinaFeighnHandel2023McCoolGroups} recently showed that for finite subgroups of $\Out(F_n)$, the equivariant McCool groups have Type VF, which is a crucial building block in our theorem.

In \cite{Andrew2022}, we prove that a free-by-cyclic group where the defining automorphism has linear growth has a finitely generated automorphism group. This is related to, but distinct from, the properties of the centraliser of that automorphism.

\subsection*{Proof Strategy}

In Proposition~\ref{invarianttree} we use Cohen and Lustig's work \cite{Cohen1999} to conclude that there is a minimal, irreducible $F_n$-tree with cyclic edge stabilisers which is preserved by $C(\Phi)$. Passing to a finite index subgroup, one obtains a map from the centraliser of $\Phi$ to the product of the outer automorphism groups of the $G_v$ whose kernel is often well understood (in our situation it is a free abelian group -- Lemma~\ref{lem:kernel_is_free_abelian}). 

It is also necessary to understand the image of this map: Proposition~\ref{projection} shows it is contained in a product of equivariant McCool groups, indexing over fewer orbits. To finish, we show that the map is virtually onto this product. To do this we use the following characterisation of centralisers in outer automorphism groups:

\begin{restatable}{lem}{FibrePreserving}
	\label{prop:centraliser_fibre_preserving}
	Let $G$ be a group and $\phi \in \Aut(G)$. Let $M_{\phi}$ be the mapping torus of $\phi$. That is, 
	$$M_{\phi} = G \rtimes_{\phi} \Z = \langle G, t \ : \ w^t = w \phi \text{ for all } w \in G \rangle.$$
	
	Then some $\chi \in \Aut(G)$ commutes with $\phi$ as an \textit{outer}   automorphism (which is to say that the commutator of the automorphisms is an inner automorphism, $[\phi, \chi] \in \Inn(G) $ ) if and only if for some $g \in G$, the map:
	$$
	\begin{array}{rcl}
		t & \mapsto & tg \\
		w & \mapsto & w\chi,
	\end{array}
	$$
	defines an automorphism of $M_{\phi}$. 
\end{restatable}

As in Convention~\ref{right}, the notation $w \chi$ indicates that the automorphism is acting on the right. We refer to the automorphisms of $M_\phi$ that take this form as \emph{fibre and coset preserving}.

	Equipped with this lemma, we can construct preimages. The mapping torus $M_\phi$ acts on the same tree, and we show in Proposition~\ref{finiteindex} that (perhaps up to finite index) the fibre and coset preserving extensions of each vertex McCool group can be assembled to form a fibre and coset preserving automorphism of the whole mapping torus.  
	The restriction to $G$ is then an automorphism $\chi$ whose outer class commutes with $\Phi$. 

A similar strategy should be viable for more groups and automorphisms, provided a sufficiently invariable action on a tree exists, and the relevant equivariant McCool groups can be understood.
	
\begin{ack}
	The first author was supported by EPSRC under grant number EP/W522417/1, and the second author by Leverhulme Trust Grant RPG-2018-058.
	We would also like to thank Ric Wade for comments on an earlier draft of this paper, and the anonymous referee for their constructive feedback and suggestions.
\end{ack}

\section{An Illustrative Example}
\label{sec:ex}

In this section we give a detailed example to illustrate Theorem~\ref{mainthm}. Consider the free group of rank 4, $F=F(a,b, \alpha, \beta)$. Let $I: \langle a,b \rangle \to \langle \alpha, \beta \rangle$ be the isomorphism sending $a$ to $\alpha$ and $b$ to $\beta$.
Let $g \in \langle a, b \rangle$ and let $\gamma=gI \in \langle \alpha, \beta \rangle$. (In what follows it will be easier if these are not proper powers, so we shall assume that). 

Now define the automorphism, $\phi$, of $F$ by: 
$$
\begin{array}{rcl}
	& \phi & \\
	a & \mapsto & \alpha \\
	b & \mapsto & \beta \\
	\alpha & \mapsto & a^g \\
	\beta & \mapsto & b^g. \\
\end{array}
$$

The outer automorphism $\Phi$ (represented by $\phi$) has linear growth, but is not a Dehn Twist. 

However, $\Phi^2$ is a Dehn Twist: 
$$
\begin{array}{rcl}
	& \phi^2 & \\
	a & \mapsto & a^g \\
	b & \mapsto & b^g \\
	\alpha & \mapsto & \alpha^{\gamma} \\
	\beta & \mapsto & \beta^{\gamma} \\
\end{array}
$$

There are three isogredience classes (genuine automorphisms in the outer class of $\Phi^2$ which are not conjugate by any inner automorphism) in $\Phi^2$ whose fixed subgroup has rank at least two. These fixed subgroups form the vertex groups of a graph of groups of which $\Phi^2$ can be realised as an (efficient) Dehn Twist and is one way to construct the graph of groups, $\GG$, below. They are:   
$$
\begin{array}{lcl}
	\Fix  \phi^2 \Ad(g^{-1})& = & \langle a,b \rangle \\
	\Fix \phi^2 & = & \langle g,\gamma \rangle \\
	\Fix \phi^2  \Ad(\gamma^{-1})& = & \langle \alpha, \beta \rangle. \\
\end{array}
$$

(We remind the reader of our Convention~\ref{right} that we place automorphisms on the right and that throughout the paper $\Ad(x)$ denotes the inner automorphism $h \mapsto x^{-1}hx$.)

\medskip

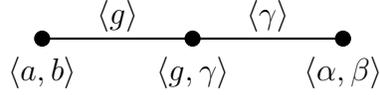
\begin{figure}
	\begin{tikzpicture}
		\draw[thick] (0,0) -- (4,0); \node at (1,0.3) {$\langle g \rangle$}; \node at (3,0.3) {$\langle \gamma \rangle$};
		\draw[fill] (0,0) circle [radius=0.1]; \node at (0,-0.45) {$\langle a,b \rangle$};
		\draw[fill] (2,0) circle [radius=0.1]; \node at (2,-0.45) {$\langle g, \gamma \rangle$};
		\draw[fill] (4,0) circle [radius=0.1]; \node at (4,-0.45) {$\langle \alpha,\beta \rangle$};
	\end{tikzpicture}
	\caption{A graph of groups $\GG$ on which $\Phi$ is realised as the root of a Dehn twist.}
	\label{fig:gog}
\end{figure}

Returning to $\Phi$, this is represented as a graph of groups automorphism, $R$, on the following graph of groups, $\GG$, shown in Figure~{\ref{fig:gog}}. See Section~\ref{sec:preliminaries} for an explanation of the formalism. Perhaps the most important thing to note at this stage is that the conjugation will take place outside the map on vertex groups.

Let the underlying graph $\Gamma$ have three vertices $u,v,w$ and two edges, $e_u=(u,v)$ and $e_w=(w,v)$. The vertex groups are $G_u=\langle a,b \rangle$, $G_v=\langle g, \gamma \rangle$ and $G_w=\langle \alpha, \beta \rangle$, the edge groups are infinite cyclic, with $G_{e_u}= \langle g \rangle$ and $G_{e_w}=\langle \gamma \rangle$ 
Then $\pi_1(\GG) \cong G_u*G_w \cong F$. 

The underlying graph map $R_{\Gamma}$ interchanges $u$ and $w$, and the edges $e_u$ and $e_w$, and fixes $v$. The vertex group isomorphisms are then: 
\begin{align*}
	R_u: G_u \to G_w & ;  R_u(a) = \alpha, R_u(b) = \beta \\
	R_v: G_v \to G_v & ;  R_v(g) = \gamma, R_v(\gamma) = g \\
	R_w: G_w \to G_u & ;  R_w(\alpha) = a, R_w(\beta) = b.
\end{align*}

The edge group maps should interchange the generators $g$ and $\gamma$, and set $\delta_{e_u}=1$, $\delta_{\overline{e}_u}=1$, $\delta_{e_w}= 1$ and $\delta_{\overline{e}_w}=g$, inducing the conjugation of the third and fourth generators. 
It is straightforward to check that the automorphism $R$ induces on $\pi_1(\GG,v)$ is  $\phi$. The automorphism $R^2$ is a Dehn Twist of $\GG$.

There are therefore three equivariant McCool groups: 		
$$
\begin{array}{rcl}
	\MC(G_u ; \langle g \rangle; R^2_u) & = & \MC(G_u ; \langle g \rangle; 1_{G_u}) \\
	\MC(G_v ; \{ \langle g \rangle, \langle \gamma \rangle \}; R_v) & = & 1 \\
	\MC(G_w ; \langle \gamma \rangle; R^2_w) & = & \MC(G_w ; \langle \gamma \rangle; 1_{G_w}), \\
\end{array}
$$
which follows since both $R^2_u$ and $R^2_w$ are inner, and since the automorphisms of a free group of rank 2 which preserve the conjugacy classes of a basis are all inner.

A graph of groups for $M = F \rtimes_\phi \Z$ is shown in Figure~\ref{fig:gog2}.

\begin{figure}
	\begin{tikzpicture}
		\draw[thick] (0,0) -- (2,0); \node at (1,0.3) {$\langle g, t^2 \rangle$};
		\draw[fill] (0,0) circle [radius=0.1]; \node at (0,-0.45) {$\langle a,b, t^2 \rangle$};
		\draw[fill] (2,0) circle [radius=0.1]; \node at (2,-0.45) {$\langle g, \gamma, t \rangle$};
	\end{tikzpicture}
	\caption{A graph of groups $\mathcal{M}$ for the free-by-cyclic group $M$.}
	\label{fig:gog2}
\end{figure}
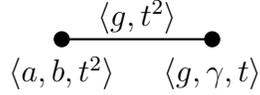

By Theorem~\ref{mainthm}, the centraliser $C(\Phi)$ has a finite index subgroup $C_0(\Phi)$ which maps onto 
$$
\MC(G_u ; \langle g \rangle; R^2_u)  \times \MC(G_v ; \{ \langle g \rangle, \langle \gamma \rangle \}; R_v) \cong \MC(G_u ; \langle g \rangle; 1_{G_u}).
$$			

To demonstrate how to construct pre-images, let $\psi_u$ be any automorphism of $G_u=\langle a,b \rangle$ which fixes $g$ (every element of the McCool group $\MC(G_u ; \langle g \rangle; 1_{G_u})$ has such a representative). Write $\psi_w = I^{-1} \psi_u I$ for the corresponding automorphism of $G_w$ (formally changing every $a$ to $\alpha$ and so on), which fixes $\gamma$. 

Then if we define $\psi = \psi_u*\psi_w$ to be $\psi_u$ on $G_u$ and $\psi_w$ on $G_w$, its outer class $\Psi$ commutes with $\Phi$. 
Note that 
if $\psi_u$ (fixing $g$) and $\psi'_w$ (fixing $\gamma$) are not `conjugate' by $I$, then $\Psi$ 
commutes with $\Phi^2$ but not $\Phi$. 

The kernel in the short exact sequence 
does not consist of all Dehn Twists on $\GG$. Consider the Dehn Twist $D$ on $\GG$ given by:
$$
\begin{array}{rcl}
	& D & \\
	a & \mapsto & a^{g^r} \\
	b & \mapsto & b^{g^r} \\
	\alpha & \mapsto & \alpha^{\gamma^s} \\
	\beta & \mapsto & \beta^{\gamma^s} \\
\end{array}
$$ 

Varying $r$ and $s$ describes a free abelian subgroup of rank 2, accounting for all the Dehn Twists on $\GG$.

Then, conjugating $\phi$ by $D^{-1}$, we see
$$
\begin{array}{rcl}
	& D \phi D^{-1} & \\
	a & \mapsto & \alpha^{\gamma^{r-s}} \\
	b & \mapsto & \beta^{\gamma^{r-s}} \\
	\alpha & \mapsto & a^{g^{s-r+1}} \\
	\beta & \mapsto & b^{g^{s-r+1}} \\
\end{array}
$$ 

So $D$ does not always centralise $\Phi$ (even as an outer automorphism), although it does centralise $\Phi^2$. They commute precisely when $r=s$, giving a rank one subgroup. The short exact sequence of Theorem~\ref{mainthm} is
\[1 \longrightarrow \Z \longrightarrow C_0(\Phi) \longrightarrow \MC(G_u;\{ \langle g \rangle \} ; 1_{G_u}) \longrightarrow 1.\]


We can also apply this theorem to the structure of the centraliser of $\Phi^2$ (represented by the Dehn Twist $R^2$ on $\GG$). The exact sequence, agreeing with \cite[Theorem 3.8]{CentralisersDehnTwists}, is: 
\[1 \longrightarrow \Z^2 \longrightarrow C_0(\Phi^2) \longrightarrow \MC(G_u;\{ \langle g \rangle \} ; 1_{G_u}) \times \MC(G_w;\{ \langle \gamma \rangle \} ; 1_{G_w}) \longrightarrow 1.\]

Note that since $\Phi$ acts by permuting the only two edges of $\Gamma$, $C_0(\Phi)$ is an index 2 subgroup of $C(\Phi) = \langle C_0(\Phi), \Phi \rangle$. Since $\Phi$ commutes with $\Phi^2$, the same is true of $C(\Phi^2)$.

Although the difference in the ranks of the kernels already imply that $C_0(\Phi)$ is not commensurable with $C_0(\Phi^2)$, for many choices of $g \in \langle a,b \rangle$ the right hand terms are also not commensurable. For instance, let $g=a^{-1}b^{-1}ab$ so that $\MC(G_u;\{ \langle g \rangle \} ; 1_{G_u})$ is precisely the mapping class group of a (once punctured) torus. Then the right hand side of the exact sequence for $\Phi$ contains $F_2$ with finite index, whereas for $\Phi^2$ it contains $F_2 \times F_2$ with finite index.

Moreover with this $g=a^{-1}b^{-1}ab$, the graph of groups from Figure~\ref{fig:gog} encodes a surface: take a torus with a single boundary component for $\langle a,b \rangle$, another torus with a single boundary component for $\langle \alpha, \beta \rangle$  and a sphere with 3 boundary components (a pair of pants) for $\langle g, \gamma \rangle$. Now glue the boundary circles of the tori to \textit{different} boundary circles of the 3-holed sphere. The resulting surface is a genus 2 orientable surface with a single boundary component. In this situation, $\Phi$ can be realised as a homeomorphism of that surface and $\Phi^2$ is the product of two Dehn twists along disjoint curves, corresponding to $g$ and $\gamma$.

\section{Preliminaries}
\label{sec:preliminaries}

\subsection*{Graphs of groups}
	
We follow Bass \cite{Bass1993} and define a graph of groups $\GG$ to consist of a graph $\Gamma$ (as defined by Serre \cite{Serre2003}, with edges in pairs $\{e,\overline{e}\}$ and maps $\iota(e)$ and $\tau(e)$ indicating the initial and terminal vertices of $e$), together with groups $G_v$ for every vertex and $G_e = G_{\overline{e}}$ for every edge, and monomorphisms $\alpha_e:G_e \to G_{\tau(e)}$ for every edge.
	
The path group $P(\GG)$ is the group generated by the vertex groups and the edges of $\GG$, subject to relations $e\alpha_e(g)\overline{e}=\alpha_{\overline{e}}(g)$ for $g \in G_e$. Note that taking $g=1$ this gives that $e^{-1}=\overline{e}$, as expected.
	
A \emph{path} (of length $n$) in $P(\GG)$ is a sequence $g_0e_1g_1 \dots e_ng_n$, where each $e_i$ has $\iota(e_i)=v_{i-1}$ and $\tau(e_i)=v_i$ for some vertices $v_i$ (so there is a path in the graph), and each $g_i \in G_{v_i}$. A \emph{loop} is a path where $v_0=v_n$. The relations in $P(\GG)$ put an equivalence relation on the set of paths, analogous to homotopy, which we work up to throughout. 
	
The \emph{fundamental group of $\GG$} at a vertex $v$ is the set of loops in $P(\GG)$ at $v$, inheriting the multiplication of $P(\GG)$ (working up to the above equivalence). It is denoted $\pi_1(\mathcal{G},v)$. Thus we think of $\pi_1(\mathcal{G},v)$ as a subgroup of $P(\GG)$.

One can also define a tree with an action by $\pi_1(\GG,v)$; the structure theorem for Bass--Serre theory asserts that the quotient graph of groups associated with such an action is sufficient to reconstruct it (up to isomorphisms). We will take the following construction of the quotient graph of groups for a given action on a tree.

\begin{defn}
	\label{def:quotient_gog}
	Suppose a group $G$ acts on a tree $T$. The \emph{quotient graph of groups} $\GG$ for this action has the quotient graph for the action as its underlying graph. We think of the edges and vertices of $\GG$ as a subset of the edges and vertices of $T$, choosing a representative for each $G$-orbit. (The choice should respect the edge inversion: if $e$ is chosen, so is $\overline{e}$.)
	
	The vertex groups are vertex stabilisers $G_v$ of the representative of each $G$-vertex orbit, and the edge groups are the edge stabilisers $G_e$, again of the representative of each edge orbit.

	Then, for each edge $e$ of $\GG$, there are group elements $g_e^{-}$ and $g_e^{+}$ of $G$ such that $\iota(e) g_e^{-}$ and $\tau(e) g_e^{+}$ are our chosen orbit representatives in $T$. 
	The choice of $g_e^{-}, g_e^{+}$ is not canonical, but one choice is made from the start, and should be compatible with the edge inversion: $g_{\overline{e}}^{-} = g_e^{+}$.

	The edge monomorphisms $\alpha_e$ are given by the inclusion of the edge stabiliser into the vertex stabiliser, followed by conjugation by $g_e^{+}$: \[
	G_e \subseteq G_{\tau(e)} \to G_{\tau(e)}^{g_e^{+}} = G_{\tau(e)g_e^{+}}. \]
\end{defn}

It is common to require representatives of vertex and edge orbits that form subtrees of $T$, in which case at least one of $g_e^{-}$ and $g_e^{+}$ can be chosen to be trivial. However, we found it more convenient to adopt this more symmetric notation.

A result we will need later is the following: 

\begin{prop}
	\label{extendhoms}
	Let $\GG$ be a (connected) graph of groups and $v$ a vertex of the underlying graph. Let  $P(\GG)$ be the associated path group and $\pi_1(\GG,v)$ the fundamental group at $v$, thought of as a subgroup of $P(\GG)$. 
	
	Then for any homomorphism, $\rho: \pi_1(\GG,v) \to Q$, where $Q$ is any group, there exists a homomorphism $\hat{\rho}: P(\GG) \to Q$, extending $\rho$ in the sense that the following diagram commutes:

	\begin{center}
		\begin{tikzcd}
		P(\GG) \ar{r}{\hat{\rho}}  &Q \\
			\pi_1(\GG,v) \ar{u}{\subseteq}  \ar{ru}[swap]{\rho} & 
		\end{tikzcd}
	\end{center}
	
\end{prop}
\begin{proof}
Take as a generating set for $P(\GG)$ the collection of all vertex groups $G_v$ along with all edges $e$. Pick a maximal tree $S$ (so as not to confuse it with the Bass-Serre tree) of $\GG$. For every vertex $u$ let $\sigma_u$ be the unique reduced path in $S$ from $v$ to $u$. We define $\hat{\rho}$ as follows: 
$$
\begin{array}{rcll}
	\hat{\rho}(e) & = & 1 &\text{for all } e \in S, \\
	\hat{\rho}(e) & = & \rho(\sigma_{\iota(e)} e {\sigma_{\tau(e)}}^{-1}) &\text{for all } e \not\in S, \\ 
	\hat{\rho}(g_u) & = & \rho(\sigma_u g_u \sigma_u^{-1}) &\text{for all } g_u \in G_u. 
\end{array}
$$

(Note that the second condition applied to all edges implies the first, which we have only included for clarity).

To check that $\hat{\rho}$ defines a homomorphism to $Q$ we simply need to check that all the relations in $P(\GG)$ are sent to the trivial element by $\hat{\rho}$. There are two kinds of relations; those within a vertex group and the edge relations. For the first, if $g_uh_uk_u=1$ then,
$$
\hat{\rho}(g_u h_u k_u) = \rho(\sigma_u g_u h_u k_u \sigma_u^{-1}) = 1. 
$$
For an edge relation, if we take some $x \in G_e$ and consider the relation $e \alpha_e(x) e^{-1} = \alpha_{\overline{e}}(x)$. Let $u = \iota(e), w=\tau(e)$. Then in $\pi_1(\GG,v)$ we have the corresponding relation
$$
\sigma_u e \alpha_e(x) e^{-1} \sigma_u^{-1} = \sigma_u \alpha_{\overline{e}}(x) \sigma_u^{-1}, 
$$
and hence
$$\rho(\sigma_u e \alpha_e(x) e^{-1} \sigma_u^{-1} ) = \rho(\sigma_u \alpha_{\overline{e}}(x) \sigma_u^{-1} ).
$$

But now, $$
\begin{array}{rcl}
\hat{\rho}(e \alpha_e(x) e^{-1}) &  = &  \hat{\rho}(\sigma_u e \sigma_w^{-1} \sigma_w \alpha_e(x) \sigma_w^{-1} \sigma_w e^{-1} \sigma_u^{-1}) \\ \\& = & \rho(\sigma_u e \sigma_w^{-1}) \rho(\sigma_w \alpha_e(x) \sigma_w^{-1}) \rho(\sigma_w e^{-1} \sigma_u^{-1}) \\ \\ & = & \rho(\sigma_u e \alpha_e(x) e^{-1} \sigma_u^{-1} ) \\  \\ & = &  \rho(\sigma_u \alpha_{\overline{e}}(x) \sigma_u^{-1} ) \\ \\ & = & \hat{\rho}( \alpha_{\overline{e}}(x) ).
\end{array} 
$$

This verifies that $\hat{\rho}$ is a well-defined homomorphism from $P(\GG)$ to $Q$. The fact that it extends $\rho$ is immediate from the definition; we simply take the generators of $\pi_1(\GG,v)$ to be the elements $\sigma_{\iota(e)} e {\sigma_{\tau(e)}}^{-1}$ where $e$ ranges over the edges of $\GG$, and $\sigma_u g_u \sigma_u^{-1}$, where $g_u$ ranges over the vertex groups, $G_u$.
\end{proof}

\subsection*{Automorphisms of graphs of groups}

Given a graph of groups $\GG$ an automorphism $F$ of $\GG$ consists of a graph automorphism $F_\Gamma$, group isomorphisms $f_v:G_v \mapsto G_{F_\Gamma(v)}$ and $f_e:G_e \mapsto G_{F_\Gamma(e)}$, and elements $\delta_e \in G_{F_\Gamma(\tau(e))}$ for each edge, satisfying $f_{\tau(e)}(\alpha_e(g))=\delta_e^{-1}(\alpha_{F_\Gamma(e)}(f_e(g)))\delta_e$. 
	
An automorphism of $\GG$ induces a group automorphism of the path group, sending elements of $G_v$ to their image in $f_v(G_v)$ and edges $e$ to $\delta_{\overline{e}}^{-1} F_\Gamma(e) \delta_e$, and in turn an isomorphism from $\pi_1(\GG,v)$ to $\pi_1(\GG,F(v))$. If $F(v)=v$ this is an element of $\Aut(\pi_1(\GG,v))$; otherwise different choices of a path in $P(\GG)$ joining $v$ and $F(v)$ give elements that differ in $\Aut(\pi_1(\GG,v))$ by conjugation (by the loop formed by concatenating the paths) and so this gives a well-defined element of $\Out(\pi_1(\GG,v))$.

As one might expect, automorphisms of graphs of groups form a group under composition. The graph and group maps compose as normal; the $\delta$-values of the new map can be seen by considering the image of an edge (as an element of the path group) and reading off the `non-edge' terms. 

Moreover, any automorphism of $\GG$ induces an automorphism of the associated Bass--Serre tree $T$. Here, by an automorphism of $T$ we mean a graph map on $T$ which sends vertices to vertices and edges to edges. An automorphism of $T$ arising in this way will permute $G$-orbits. 
	
(Note that Bass allows for a more general notion of isomorphism, inducing additional conjugations on the fundamental group; these do not induce extra outer automorphisms, so we do not need to consider them here.)
	
\subsection*{Dehn Twists and Linear Growth outer automorphisms of free groups} One kind of automorphism of a graph of groups is a \emph{Dehn Twist}, where the graph map and group homomorphisms are required to be trivial, and the values $\delta_e$ are required to be the image $\alpha_e(g_e)$ of an element $g_e \in Z(G_e)$. Together with the edge relations, these give \emph{twistors} $z_e = \delta_e\delta_{\overline{e}}^{-1}$ for each edge $e$ (and $z_e=z_{\overline{e}}^{-1}$) where as an element of the path group $e \mapsto z_ee$. The induced element of $\Out(\pi_1(\GG,v))$ depends only on these twistors, and not on the values of $\delta_e$ used to define them; provided $G$ is centreless the subgroup they comprise is isomorphic to a quotient of the direct product of the centres $Z(G_e)$  taken over the geometric edges. 
	
The maps and elements do not uniquely determine the Dehn Twist, as an automorphism, and following \cite{KrsticLustigVogtmann}, we say that a graph of groups automorphism \emph{represents} a Dehn Twist if it induces the same element of $\Out(G)$ as the Dehn Twist. For example, replacing the group homomorphisms with inner automorphisms, adjusting the values of the $\delta$ appropriately, realises the same (outer) automorphism of the fundamental group. In fact, by \cite[Proposition 4.6(1)]{KrsticLustigVogtmann}, this is the only way to get alternative representatives in the case we are concerned with, where $G$ is a finitely generated free group and $\GG$ has edge groups maximal cyclic in adjacent vertex groups.
	
	
Cohen and Lustig \cite{Cohen1999} define an \emph{efficient Dehn twist} on a graph of groups $\GG$ with $\pi_1(\GG)=F_n$. This adds certain conditions which amount to saying that the $F_n$-action is very small and that there are no `unused edges'; in particular this implies that the graph of groups is finite, vertex groups are free groups of rank at least $2$, and edge groups are maximal infinite cyclic. For our purposes it is enough to note that every Dehn twist has an efficient realisation. 
	
Suppose $\GG$ is a graph of groups, and $D$ a Dehn Twist on $\GG$. Say that a graph of groups automorphism $R$ is a \emph{root of $D$} if there is some $k$ such that $R^k$ represents $D$, and $R_e(z_e)=z_{R_\Gamma(e)}$ for all edges $e$ of $\GG$.

Given an element $\Psi$ of $\Out(F_n)$, consider the effect of iterating $\Psi$ on elements of $F_n$ by fixing a basis and considering the length of the shortest representative of the conjugacy class $g\Psi^k$. Say $\Psi$ is linearly growing if the growth in this sense of every element is bounded above by $Ak$, and if the growth of some element is bounded below by $Bk$, for some constants $A$ and $B$. (In general, elements of $\Out(F_n)$ may be exponentially growing or polynomially growing, with the possible degrees bounded above by $n$ -- see \cite{Bestvina1992}, \cite{Bestvina2000} and \cite{Levitt2009}.) 

However, the property of these automorphisms we make use of throughout this paper -- which may be taken as the definition -- is the following.
	
\begin{thm}[{\cite[Proposition 5.3]{KrsticLustigVogtmann}}]
	\label{lineartodehn}
	Let $\Phi$ be a linearly growing element of $\Out(F_n)$. Then $\Phi$ is realised by a root of an efficient Dehn Twist on a graph of groups $\GG$.
\end{thm}
	
\begin{rem}
	This statement is a combination of several results: that all polynomially growing elements of $\Out(F_n)$ have a power which is \emph{UPG}; that UPG and linearly growing elements of $\Out(F_n)$ are represented by Dehn twists, and that outer automorphisms having a power represented by a Dehn twist can be realised by a root of a Dehn Twist {\cite[Proposition 5.3]{KrsticLustigVogtmann}}. These results are discussed in detail in \cite[Theorem~2.4.6]{Andrew2022}.

\end{rem}
	
Note that the graph of groups automorphism realising $\Phi$ might only have a power representing the Dehn Twist; in fact this must be the case if it involves automorphisms of a vertex group which are finite order as outer automorphisms but not as automorphisms.

\subsection*{Invariant Trees}

Dual to a graph of groups is the Bass--Serre tree. All the actions on trees we consider will be \emph{minimal} -- admitting no proper invariant subtree -- and \emph{irreducible}, acting without fixing an end on a tree that is not a line (equivalently, containing a non-abelian free group acting freely).

Given a group, $G$, acting on a tree, $T$, we can define the translation length function, $\|.\|_T: G \to \Z$, given by $\|g\|_T:= \inf_{x \in T} d_T(x, xg)$. For minimal and irreducible actions,
this function determines both the tree $T$ and the action of $G$ on it (see \cite[Theorem 7.13(b)]{AlperinBass1984} and \cite[Theorem 3.7]{CullerMorgan1987Rtrees}). . 

Note that the translation length function is constant on conjugacy classes, thus the following definitions make sense.

\begin{defn}
	 Suppose we have a group $G$ acting without inversions on a tree $T$. We set %
	 \begin{itemize}
	 	\item $\Out^T(G) = \{ \Phi \in \Out(G) \ : \ \|g\Phi \|_T = \|g\|_T \text{ for all } g \in G    \}$, 
	 	\item $\Aut^T(G)$ to be the full pre-image of $\Out^T(G)$ in $\Aut(G)$. 
	 	
	 	That is,   $\Aut^T(G) = \{ \phi \in \Aut(G) \ : \ \|g\phi \|_T = \|g\|_T \text{ for all } g \in G    \}$. 
	 \end{itemize}
\end{defn}
	
These automorphisms can be studied through the action on the invariant tree (see for instance \cite{Naomi2021} for a detailed discussion of how to derive this from the literature):
	
\begin{prop}
	Let $G$ act on a tree, $T$. Then $\Aut^T(G)$ also acts on $T$, and extends the $G$-action, where each element of $G$ is identified with the inner automorphism it induces. 
\end{prop}	
	
Another perspective comes from the work of Bass and Jiang:

\begin{prop}[{\cite[Theorem 4.1]{BassJiang}}]
	\label{prop:gog_auts_act}
	Suppose $G$ acts on a tree $T$ with quotient graph of groups $\GG$. An outer automorphism $\Psi$ of $G$ is contained in $\Out^T(G)$ if and only if it has a representative which can be realised as an automorphism of $\GG$.
\end{prop}

We will use both perspectives -- automorphisms acting on a tree, and automorphisms of the quotient graph of group -- throughout. 

\subsection*{Fibre and Coset Preserving Automorphisms}

Here we prove the elementary lemma which allows us to relate centralisers of outer automorphisms to automorphisms of the mapping torus.


\FibrePreserving*

\begin{proof}
	If the map above defines an endomorphism of $M_{\phi}$, then it is straightforward to check it will be bijective (and the inverse will be given by the map $t \mapsto t (g^{-1} \chi^{-1}), w \mapsto w \chi^{-1}$). Hence it is sufficient to check that the map above defines an endomorphism.

	Therefore, by an application of Von Dyck's Theorem \cite[Theorem 2.2.1]{RobinsonTextbook}, the map given above defines an automorphism of $M_{\phi}$  if and only if for all $w \in G$ we have that:
	$$
	(w \chi)^{tg} = w \phi \chi,
	$$
	
	which is equivalent to 
	%
	%
	%
	$$
	\begin{array}{rcll}
		&(w \chi \phi)^g & = & w \phi \chi  \\
		\Leftrightarrow \quad & \chi \phi \Ad(g) & =&  \phi \chi  \\
		\Leftrightarrow \quad & \Ad(g)  & = & \phi^{-1} \chi^{-1} \phi \chi, 
	\end{array}
	$$
	
	where $\Ad(g)$ denotes the inner automorphism defined by $g$. 
	
	Hence the map is an automorphism of $M_{\phi}$ if and only if  $[\phi, \chi] = \Ad(g) \in \Inn(G) $. 
\end{proof}

\begin{rem}
	We note that in the case that $G$ has a centre, the element $g$ in Lemma~\ref{prop:centraliser_fibre_preserving} is not unique; it is only well defined up to a coset of the centre. However, Lemma~\ref{prop:centraliser_fibre_preserving} works with any and all of these choices.
\end{rem}

Here it can help to make the following definition: 

\begin{defn}
	\label{def:coset_fibre_preserving}
	Let $M$ be a group, and $G$ be a normal subgroup of $M$ -- for instance, when $M$ is a mapping torus of $G$. An element $\chi$ of $\Aut(M)$ is \emph{fibre and coset preserving for $G$} if it preserves every coset of $G$ in $M$. In particular, such a $\chi$ restricts to an automorphism of $G$. 
\end{defn}

\begin{rem}
	Lemma~\ref{prop:centraliser_fibre_preserving} can then be restated as saying that $\chi$ commutes with $\phi$ as an outer automorphism if and only if it extends to a fibre and coset preserving automorphism of $M_{\phi}$ -- we use this equivalence repeatedly.
\end{rem}

\subsection*{Equivariant McCool Groups} 
An important input to our result is the study of \emph{equivariant McCool groups}, carried out by Bestvina, Feighn and Handel \cite{BestvinaFeighnHandel2023McCoolGroups} for finite order elements of $\Out(F_n)$. Here we recall the definition of these subgroups, as well as their result.

\begin{defn}
	\label{def:mccool}
	Let $G$ be a group, let $\{ G_i \}$ be a finite family of subgroups of $G$ and $\Phi$ an outer automorphism of $G$. 
	
	Then, the $\Phi$-equivarant McCool group relative to the $G_i$ (or simply the equivariant McCool group) $\MC(G;\{ G_i \} ; \Phi)$ consists of those outer automorphisms $\Psi$ of $G$ such that: 
	\begin{enumerate}[(i)]
		\item For each $G_i$, there is a $\psi_i \in \Psi$ (a genuine automorphism in the outer class of $\Psi$) such that;
		\begin{itemize}
			\item $\psi_i(G_i)= G_i$, 
			\item $\psi_i|_{G_i}$ is the identity map
		\end{itemize}
		\item $\Psi$ commutes with $\Phi$. 
	\end{enumerate}
\end{defn}

	One can also define the generalised McCool group by 
	dropping the condition that $\psi_i$ restricts to the identity map on $G_i$. This is the point of view in \cite{BestvinaFeighnHandel2023McCoolGroups}. 
	
	The `usual' McCool groups -- without equivariance -- can be recovered by putting $\Phi$ equal to the identity.
	%

The theorem we need is:

\begin{thm}[{\cite[Theorem 1.2]{BestvinaFeighnHandel2023McCoolGroups}}]
	\label{mccool}
	Let $F$ be a finitely generated free group, let $\{ G_i\}$ be a finite family of finitely generated subgroups and let $\Phi \in \Out(F)$ be a finite order outer automorphism. Then  $\MC(F ;\{ G_i \} ; \Phi)$ is of type VF. 
\end{thm}
\begin{rem}
	In fact, the Theorem in \cite{BestvinaFeighnHandel2023McCoolGroups} is stated for generalised McCool groups, but it follows that the same is true for McCool groups. This is because Corollary 1.6 of \cite{Guirardel2015} states that any McCool group (in a free group, or even a toral relatively hyperbolic group) is equal to the McCool group of a finite family of cyclic groups. And since cyclic groups have finite automorphism group, this is commensurate with the generalised McCool group of the same collection of cyclic groups. 
	
	The equivariant McCool groups are then simply intersections with a centraliser. 
\end{rem}


\section{Proof of the theorem} 

%

\subsection*{Notation and Actions on Invariant Trees}

We set the following notation throughout: 
\begin{Not} \

\label{notationcent}	
	
	\begin{itemize}
	\item $F_n$ is a free group of rank $n$, 
	\item $\Phi \in \Out(F_n)$ is a fixed linearly growing outer automorphism,
	\item $C(\Phi)$ is the centraliser of $\Phi$ in $\Out(F_n)$, 
	\item $\widehat{C}(\Phi)$ is the full pre-image of $C(\Phi)$ in $\Aut(F_n)$. 
	\item Set $G=F_n$ and $M=G \rtimes_{\Phi} \Z$, the corresponding mapping torus with monodromy $\Phi$. (Note that the isomorphism type does not depend on the chosen representative $\varphi$ of $\Phi$)
	\item By abuse of notation, we will write $G \leq M \leq \widehat{C}(\Phi)$, thinking of $M$ as the full pre-image of $\langle \Phi \rangle$ in $\Aut(F_n)$ and identifying $F_n$ with $\Inn(F_n)$.
	\end{itemize}
\end{Not}		

Just as in \cite{CentralisersDehnTwists}, our proof relies on a theorem of Cohen and Lustig:

\begin{thm}[{\cite[Proposition 7.1(a)]{Cohen1999}}]
	\label{thm:CohenLustigConjugacy}
	Suppose $\Psi \in \Out(F_n)$ is represented by an efficient Dehn twist, based on $\GG$ with twistors $\{z_e\}$. Then the centraliser $C_{\Out(F_n)}(\Psi)$ consists of outer automorphisms induced by graph of groups automorphisms of $\GG$ which preserve twistors.
\end{thm}

Any element centralising $\Phi$ will centralise any power of $\Phi$, so any element of $C(\Phi)$ can be realised as a graph of groups automorphism of $\GG$ which preserves twistors. The following proposition is largely a precise statement of some of the consequences of this theorem.

\begin{prop}
	\label{invarianttree}
	Given a linear growth outer automorphism $\Phi$ there is a minimal, co-compact $G=F_n$ tree $T$ with $\widehat{C}(\Phi) \leq \Aut^T(F_n)$. In particular, $\widehat{C}(\Phi)$ acts on $T$ -- perhaps inverting some edges -- and this action is compatible with the inclusions $G \leq M \leq \widehat{C}(\Phi)$. After subdividing as necessary,  the $G$- and $M$-stabilisers of this action have the following properties:
	\begin{itemize}
		\item Each $G$-vertex stabiliser $G_v$ is a finitely generated free group of rank at least 2, unless it is a subdivision vertex, in which case it is maximal infinite cyclic,
		\item Each $M$-vertex stabiliser $M_v$ is virtually $G_v \times \Z$,
		\item Each $G$-edge stabiliser $G_e$ is maximal infinite cyclic, and in particular malnormal.
		\item Each $M$-edge stabiliser $M_e$ is $G_e \rtimes \Z$ and hence is either free abelian of rank 2, or the Klein bottle group. 
	\end{itemize}
	In fact, each vertex stabiliser $M_v$ may be written as $G_v \rtimes_{\Phi^{t(v)} \mid_{G_v}} \Z$, where $t(v)$ is the minimal positive power such that $\Phi$ has a representative preserving $G_v$ (equivalently, every representative preserves the $G$-orbit of $v$, and some representative stabilises it), and the restriction to $G_v$ of any such representative induces a finite order outer automorphism $\Phi^{t(v)} \mid_{G_v}$.

\end{prop}
\begin{proof}
	By Theorem~\ref{lineartodehn}, there is a graph of groups $\GG$ on which $\Phi$ is realised as a root of an efficient Dehn twist. The tree $T$ is the Bass--Serre tree for this graph of groups. Any element centralising $\Phi$ also centralises all powers $\Phi^k$. By considering the power which is a Dehn Twist, we see from Theorem~\ref{thm:CohenLustigConjugacy} that every element of $C(\Phi)$ can be realised as an automorphism of $\GG$. By Proposition~\ref{prop:gog_auts_act}, this means $C(\Phi)$ is contained in $\Out^T(G)$, and so $\widehat{C}(\Phi)$ is contained in $\Aut^T(F_n)$.
	
	The equivalence of the definitions of $t(v)$, as well as the fact that $\Phi^{t(v)} \mid_{G_v}$ is a well defined outer automorphism of $G_v$ follow from the fact every vertex group $G_v$ is equal to its own normaliser. This is implied by the observation that every edge group of $\GG$ is a proper subgroup of the adjacent vertex groups. (In particular, note that any two representatives of $\Phi$ that both preserve $G_v$ must differ by an inner automorphism induced by an element of $G_v$.)
	
	The statements about $G$-stabilisers follow from the definition of an efficient Dehn twist; the corresponding statements about $M$-stabilisers follow from~\cite[Proposition 2.6]{Dahmani2016}), 
	 which gives that $M_v \cong G_v \rtimes \langle \phi^{t(v)}g \rangle$. Since $\Phi$ is a root of a Dehn twist, there is some power $\Phi^k$ which is realised by a representation of a Dehn twist on $\GG$, and in particular restricts to an inner automorphism at every vertex group. The restriction $\Phi^{t(v)} \mid_{G_v}$ is a root of this, and so must be finite order as an outer automorphism.
\end{proof}


\begin{Not}
	Henceforth, we fix the notation $t(v)$ (for the minimal power of $\Phi$ preserving the vertex orbit) and $\Phi^{t(v)} \mid_{G_v}$ introduced in Proposition~\ref{invarianttree}.
\end{Not}

This allows us to view $C(\Phi)$ as a subgroup of those outer automorphisms of $F_n$ preserving the action on a tree encoded by $\GG$, which is described in \cite[Theorem 8.1]{BassJiang}. As observed in \cite[Theorem 2.11]{CentralisersDehnTwists}, when (as in our case) the graphs of groups have malnormal edge groups, this can be simplified.

\begin{thm}[{\cite[Theorem 8.1]{BassJiang}} and {\cite[Theorem 2.11]{CentralisersDehnTwists}}]
	\label{bassjiang}
	Suppose $\GG$ is a graph of groups corresponding to a minimal action on a tree, $T$, without fixed ends, and with underlying graph $\Gamma = T/G$. 
	
	Suppose further that the edge groups are malnormal as subgroups of $G=\pi_1(\GG,v_0)$. Let $\Out^T(G)$ be the subgroup of $\Out(G)$ preserving the length function of this action (equivalently, the splitting of $G$ indicated by $\GG$). 
	Then $\Out^T(G)$ has a filtration by normal subgroups, \[\Out^T(G) \trianglerighteq \Out_0^T(G) \trianglerighteq K \trianglerighteq 1.\] such that \begin{enumerate}
			\item $\Out^T(G)/\Out_0^T(G)$ is isomorphic to a subgroup of $\Aut(\Gamma)$;
			\item $\Out_0^T(G)/K$ is isomorphic to a subgroup of $\prod_{v \in V(\Gamma)} \Out(G_v)$;
			\item $K$ is the group of Dehn Twists of $\GG$.
		\end{enumerate}
	Moreover, if all inclusions of edge groups into vertex groups are proper, then $K$ is a direct product of the centres of the edge groups, $Z(G_e)$ (working over geometric edges). 
\end{thm}

(The last observation follows from \cite[Proposition 5.4]{Cohen1999}.)
%

We call the map from $\Out_0^T(G)$ to $\prod_{v \in V(\Gamma)} \Out(G_v)$, with kernel $K$, $\mu$ and refer to it frequently through our arguments. We describe its construction here, for future reference. From the perspective of a graph of groups automorphism, acting as the identity on the underlying graph, one can write down automorphisms of each vertex group; the fact that this map is given to the \emph{outer} automorphism groups is a consequence of the failure of uniqueness among representations of a given (outer) automorphism on the graph of groups.

However, we prefer to take a definition from the perspective of the $\Aut^T(G)$ action on $T$. Let $\Aut_0^T(G)$ be the full preimage of $\Out_0^T(G)$, and observe that this subgroup will preserve the $G$-orbits (equivalently, it has the same quotient graph as the original $G$-action). Consider a $\chi \in \Aut_0^T(G)$, and a vertex $v$ of $T$. There exists a $y \in G$ such that $v \cdot \chi  = v \cdot y$.

Note that for any $g$, we have that $\Ad(g) \chi = \chi \Ad(g \chi)$. If we restrict to those $g \in G_v$, and note that both $g$ and $\Ad(g)$ act in the same way on $T$, we get that,  
$$
v= v \cdot g \cdot \chi \cdot y^{-1} = v \cdot \chi \cdot (g \chi) \cdot y^{-1} = v \cdot y (g \chi) y^{-1} = v \cdot (g \chi)^{y^{-1}}. 
$$

Hence if $g \in G_v$, then $(g \chi)^{y^{-1}}$ is also in $G_v$.  
 
\begin{defn}
	\label{def:mu}
Given $\chi$ and $y$ as above, 	$ \hat{\mu}_v (\chi)$ is the automorphism of $G_v$ defined by sending $g \in G_v$ to $(g \chi)^{y^{-1}}$.

 This is an automorphism of $G_v$, since it is a restriction of an automorphism of $G$ preserving this subgroup. Note that there is a choice of elements $y$, corresponding to the stabiliser $G_v$, so this map is only well-defined up to inner automorphisms of $G_v$, giving an outer automorphism of $G_v$.
	
	Moreover, since $\chi$ also preserves edge-orbits, we know that for every edge $e$ whose initial vertex is $v$, we have a $g_e \in G_v$ such that	
	$$
	e \cdot \chi \cdot y^{-1} = e \cdot g_e.  
	$$
	
	Hence repeating the argument above for edge groups we deduce that the automorphism $ \hat{\mu}_v (\chi)$ preserves the $G_v$-conjugacy classes of the incident edge groups $G_e$.

	To define a map ${\mu}_v (\chi)$, from $\Out_0^T(G)$ to $\Out(G_v)$, observe that $\Inn(G)$ lies in the kernel of $\hat{\mu}_v$ (since the action of the inner automorphisms matches that of the corresponding group elements), which therefore factors through $\Out_0^T(G)$. The map $\mu$ is then constructed by assembling one $\mu_v$ for a representative of each vertex orbit under the $G$ (or $\Aut_0^T(G)$) action on $T$.
\end{defn} 

This leads to the following consequence of Rodenhausen--Wade's work on Dehn twists \cite[Lemma 3.6]{CentralisersDehnTwists}: 

\begin{prop}
	\label{prop:prod_of_normal_McCool}
	The image of $\mu$ is contained in the product \[ \prod_{[v] \in V(T)/G} \MC(G_v;\{G_e\}_{\iota(e)=v}).\]
\end{prop}

The discussion above shows that the image in each $\Out(G_v)$ preserves the conjugacy classes of the incident $G_e$. The fact that the image lies in the McCool group is equivalent to asserting that the induced (outer) automorphisms on the $G_e$ are trivial, which follows from the `preserves twistors' statement of Theorem~\ref{thm:CohenLustigConjugacy}. Note that the efficiency of the Dehn Twist is vital for that result.

Note that this claim concerns only a product of standard McCool groups; equivariance is dealt with later on.

We fix the following notation for the rest of the paper.

\begin{Not}
	\label{mudefined}
	
	 As in Notation~\ref{notationcent}, we have a given linear growth outer automorphism of $F_n$, denoted $\Phi$, whose centraliser is denoted $C(\Phi)$ and whose pre-image in $\Aut(F_n)$ is denoted $\widehat{C}(\Phi)$.  
	
	\begin{enumerate}[(i)]
		\item We let $T$ be the $G$-tree provided by Proposition~\ref{invarianttree}. $T$ is also an $M$-tree and a $\widehat{C}(\Phi)$-tree, with all these actions compatible with the inclusions $G \leq M \leq \widehat{C}(\Phi)$.
		\item As usual, we write $G_v$ to mean the $G$-stabiliser of a vertex $v$ of $T$, and $G_e$ to be an edge stabiliser.  
		\item We let $C_0(\Phi)$ be the subgroup of $C(\Phi)$ which preserves the $G$-orbits of $T$. That is, $C_0(\Phi) = C(\Phi) \cap \Out^T_0(G)$. This is a finite index subgroup of $C(\Phi)$. 
		\item Similarly, $\widehat{C}_0(\Phi)$ is the full pre-image of  $C_0(\Phi)$ and is the finite index subgroup of $\widehat{C}(\Phi)$ which preserves the $G$-orbits of $T$. 
	\end{enumerate}
\end{Not}



Specialising to the situation of this paper, described in Notation~\ref{mudefined}, it follows from Theorem~\ref{bassjiang} that

\begin{lem}
	\label{lem:kernel_is_free_abelian}
	The kernel of $\mu$ is a finitely generated free abelian group. 
\end{lem}

The same will be true of the intersection of $C_0(\Phi)$ and this kernel, so in order to prove our result it will be enough to understand the image of $C_0(\Phi)$ under $\mu$.


We will show that the image,  $\mu(C_0)$, is isomorphic to a finite index subgroup of the product \[ \prod_{[v] \in V(T)/M} \MC(G_v;\{G_e\}_{\iota(e)=v};  \Phi^{t(v)} \mid_{G_v} ),\] where the product is taken over representatives $v$ of vertex orbits for the action of $M$ on $T$.  (Note that that there are only finitely many conjugacy classes of the $G_e$, since the $G$-action on $T$ is co-compact). 

We will show both inclusions: first that the image is contained in 
the given product, and then the converse. 


\subsection*{Characterising the image of \texorpdfstring{$\mathbf{\mu}$}{mu}:}

%
%
%

We need to show two additional facts about the image of $\mu$: that in fact it is sufficient to take one copy of each McCool group per $M$-orbit, and that the images in each factor are contained in the centraliser of $\Phi^{t(v)} \mid_{G_v}$.

\begin{prop}
	\label{projection}

	Consider a projection map (by choosing orbit representatives): 
	
	$$
	\rho: \prod_{[v] \in V(T)/G} \Out(G_v) \to \prod_{[v] \in V(T)/M} \Out(G_v). 
	$$
	
	Then $\ker(\rho) \cap \mu(C_0(\Phi)) = \{ 1\}$. 
	
	Moreover, for each $v$, $\mu_v(C_0(\Phi))$ is a subgroup of $\MC(G_v;\{G_e\}_{\iota(e)=v};  \Phi^{t(v)} \mid_{G_v} )$. 
	
	Hence $\mu(C_0(\Phi))$ is isomorphic via $\rho$ to a subgroup of $\prod_{[v] \in V(T)/M} \MC(G_v;\{G_e\}_{\iota(e)=v};  \Phi^{t(v)} \mid_{G_v} )$. 
	
\end{prop}
\begin{proof}
	Recall that $M = F_n \rtimes_{\Phi} \Z$, and we can think of any $m \in M$ as an automorphism of $F_n$ via the conjugation action and this agrees with the inclusion $M \leq \widehat{C}(\Phi)$. Given $g \in F_n=G$ and $m \in M$, we write $g \cdot m$ for this image (note that this is, strictly speaking, $m^{-1} g m$, if we view both $g$ and $m$ as elements of $M$). 
	
	Moreover, $M$ acts on $T$, and this action can be viewed both as an extension of the action of $F_n$ and as arising from the fact $M$ is a subgroup of $\widehat{C}(\Phi)$. 
	
	Given $m \in M$, there is an isomorphism induced by $m$ between the automorphism (and outer automorphism) groups of $G_v$ and $G_{vm}=(G_v)\cdot m$. That is, take some $\varphi \in \Aut(G_v)$ and define an automorphism of $G_{vm}$ on elements $g\cdot m$ (for some $g \in G_v$) as 
	
	\begin{center}
		\begin{tikzcd}
			g\cdot m \ar{r}{m^{-1}} & g \ar{r}{\varphi} & g\varphi \ar{r}{m} & g \varphi \cdot m = (g \cdot m)  (m^{-1}  \phi  m),
		\end{tikzcd}
	\end{center}
	That is, $m^{-1}  \phi  m$ is the unique automorphism which makes the following diagram commute: 
	
	\begin{center}
		\begin{tikzcd}
			G_v  \ar{r}{m} \ar{d}{\phi} &G_{vm} \ar{d}{m^{-1} \phi m}\\
			G_v  \ar{r}{m} & G_{vm}
		\end{tikzcd}
	\end{center}

Moreover, if $\phi$ is an inner automorphism of $G_v$, then $m^{-1}  \phi  m$ will be an inner automorphism of $G_{vm}$. More precisely, if $\phi = \Ad(h)$ for some $h \in G_v$, then $m^{-1}  \phi  m = \Ad(h \cdot m)$. 

Denote by $\Ad(m)$ this map from $\Aut(G_v)$ to $\Aut(G_{vm})$ sending $\phi$ to $m^{-1}  \phi  m$. 	Then $\Ad(m)$ is an isomorphism (with inverse $\Ad(m^{-1})$). Since it preserves inner automorphisms, it induces an isomorphism between $\Out(G_v)$ and $\Out(G_{vm})$ which again we call $\Ad(m)$. 

Next we claim that the following diagram commutes for all elements $m$ of $M$ (viewed as a subgroup of $\widehat{C}(\Phi)$, and recalling that $\widehat{C}_0(\Phi)$ is normal in $\widehat{C}(\Phi)$):

\begin{center}
	\begin{tikzcd}
		\hat{C}_0(\Phi)  \ar{r}{\hat{\mu}_v} \ar{d}[swap]{\Ad(m)} &\Out(G_v) \ar{d}{\Ad(m)}\\
		\hat{C}_0(\Phi)  \ar{r}{\hat{\mu}_{v m}} &\Out(G_{v m})
	\end{tikzcd}
\end{center}

The maps $\hat{\mu}$ are as in Definition~\ref{def:mu}, the left hand $\Ad(m)$ is conjugation by $m$ within $\hat{C}_0(\Phi)$, and the right hand $\Ad(m)$ is the isomorphism defined above. To see the commutativity, consider an element $\chi$ of $\hat{C}_0(\Phi)$, and let $y$ be any element of $G$ so that $\chi\Ad(y)$ preserves $G_v$, or equivalently $\chi y$ stabilises $v$. In particular, $(\chi\Ad(y))\mid_{G_v}$ represents the image of $\chi$ under $\hat{\mu}_v$. 

The image of this map under $\Ad(m)$ is represented by $((m^{-1}\chi m)\Ad(y\cdot m))\mid_{G_{vm}}$. Notice that that $vm (m^{-1} \chi m)(m^{-1} y m) = v \chi y m = vm$. That is, $(m^{-1} \chi m)(m^{-1} y m)$ stabilises $vm$, and hence $(m^{-1} \chi m)\Ad(m^{-1} y m)$ preserves $G_{vm}$. Hence $ \hat{\mu}_{vm} (m^{-1} \chi m)$ is again represented by $((m^{-1}\chi m)\Ad(y\cdot m))\mid_{G_{vm}}$, which proves the claim.

Inner automorphisms in $\hat{C}_0(\Phi)$ are sent by $\hat{\mu}_v$ and $\hat{\mu}_{vm}$ to inner automorphisms of $G_v$ and $G_{vm}$, since the vertex groups are self-normalising. Since both maps labelled $\Ad(m)$ preserve inner automorphisms, we get the following commuting diagram: 

\begin{center}
	\begin{tikzcd}
		{C}_0(\Phi)  \ar{r}{{\mu}_v} \ar{d}[swap]{\Ad(m)} &\Out(G_v) \ar{d}{\Ad(m)}\\
		{C}_0(\Phi)  \ar{r}{{\mu}_{v m}} &\Out(G_{v m})
	\end{tikzcd}
\end{center}

But now notice that -- by definition, since $m$ is a representative of some power of $\Phi$ -- $\Ad(m)$ is the identity map on ${C}_0(\Phi) $. Hence we get a commuting triangle. 

\begin{center}
	\begin{tikzcd}
		& \Out(G_v) \ar{dd}{\Ad(m)} \\
		C_0(\Phi) \ar{ur}{\mu_v} \ar{dr}{\mu_{v m}} &\\
		& \Out(G_{vm})
	\end{tikzcd}
\end{center}

Since $\Ad(m)$ is an isomorphism, this shows that $\ker(\mu_v) = \ker(\mu_{vm})$ and hence that $\ker(\rho) \cap \mu(C_0(\Phi)) = \{ 1\}$, proving the first claim. 

For the second claim, we simply set $m$ to be any representative of  $\Phi^{t(v)}$ which fixes $v$. Hence $vm = v$, $\mu_{vm}=\mu_v$, and $\Ad(m)$ is conjugation by  $\Phi^{t(v)} \mid_{G_v}$ within $\Out(G_v)$. Then the commutativity of the triangle gives that the image of $\mu_v$ commutes with $\Phi^{t(v)} \mid_{G_v}$.

Since equivariant McCool groups arise as the intersection of a centraliser and a McCool group, this (together with our earlier fact that the image of $\mu_v$ is in a McCool group, Proposition~\ref{prop:prod_of_normal_McCool}), shows that $\mu_v(C_0(\Phi))$ lies in the equivariant McCool group $\Mc(G_v;\{G_e\};\Phi^{t(v)}\mid_{G_v})$. This proves the second claim and the final claim follows from the previous two. 
\end{proof}

Now we want to prove that $\rho(\mu(C_0(\Phi)))$ is a finite index subgroup of

\[\prod_{[v] \in V(T)/M} \MC(G_v;\{G_e\}_{\iota(e)=v};  \Phi^{t(v)} \mid_{G_v} )\]

Our strategy involves taking an element of  $\prod_{[v] \in V(T)/M} \MC(G_v;\{G_e\}_{\iota(e)=v};  \Phi^{t(v)} \mid_{G_v} )$ and realising it as a fibre-preserving automorphism of $M$. We use several ideas from \cite{Andrew2022}, though in different ways to that paper: for instance, we do not need to pass to a nearly canonical tree for our arguments here.

First we show that applying the extension of Lemma~\ref{prop:centraliser_fibre_preserving} to a finite index subgroup of each equivariant McCool group $ \MC(G_v;\{G_e\}_{\iota(e)=v};  \Phi^{t(v)} \mid_{G_v})$ produces automorphisms of $M_v$ whose outer classes have representatives that fix each incident $M_e$. (In particular, they are contained in a McCool group of $M_v$.)

\begin{prop}
	\label{prop:nice_mccool}

	Let $v$ be a vertex of $T$ not coming from a subdivision, so its $G$-stabiliser $G_v$ is a free group of rank at least 2 
	
	Let $M_v$ be the corresponding $M$ vertex stabiliser. So $M_v = G_v \rtimes_{\Phi^{t(v)} \mid_{G_v}} \Z$, and is virtually $G_v \times \Z$. 

	Let $\chi \in \Aut(G_v)$ be an automorphism whose outer class 
	lies in $ \MC(G_v;\{G_e\}_{\iota(e)=v};  \Phi^{t(v)} \mid_{G_v})$. Write $\chi_M$ to denote the automorphism of $M_v$ induced by $\chi$ via Lemma~\ref{prop:centraliser_fibre_preserving}. 
	
	Then: 
	
	\begin{itemize}
		\item $\chi_M$ induces well-defined outer automorphisms on the edge stabilisers, $M_e$, 
		\item If, for each $M_e$ which is isomorphic to a Klein bottle group, the outer automorphism induced by $\chi_M$ is trivial, then for all edges $e$ (without condition) there exists a $h_e \in G_v$ such that $w \chi_M = w^{h_e}$ for all $w \in M_e$. 
	\end{itemize}
Hence there exists a finite index subgroup, $N$, of 	$ \MC(G_v;\{G_e\}_{\iota(e)=v};  \Phi^{t(v)} \mid_{G_v})$ such that for any $\chi \in \Aut(G_v)$ whose outer class belongs to $N$ we have that there exist elements $h_e \in G_v$ with $w \chi_M = w^{h_e}$ for all $w \in M_e$.

%
%
%
%
%
%
%
\end{prop}

%

\begin{proof}
	
		Denote by $t$ a pre-image in $M_v$ of the generator of the infinite cyclic quotient, so that conjugation by $t$ induces the (outer) automorphism $\Phi^{t(v)} \mid_{G_v}$.

		Consider some edge, $e$, and let $a=a_e$ be a generator for $G_e$. Then $M_e$ will be generated by $a$ and $t^k g$, where $0 \neq k \in \N$ and $g \in G_v$ both depend on $e$. Note that $G_e$ is a normal subgroup of $M_e$ and hence $a^{t^kg} = a^{\pm 1}$, and the corresponding $M$-edge stabiliser $M_e$ is a Klein bottle group precisely when $a^{t^kg} = a^{-1}$. 
		
		By hypothesis,  the outer class of $\chi$ belongs to $ \MC(G_v;\{G_e\}_{\iota(e)=v};  \Phi^{t(v)} \mid_{G_v})$, and hence there exists $h=h_e \in G_v$ such that $a \chi = a^{h}$. Since $\chi_M$ agrees with $\chi$ on $G_v$, we also get that  $a \chi_M = a^{h}$. 
		
		Define $\chi_e$ to be $\chi_M \Ad(h^{-1})$, lying in the same outer class as $\chi_M$, but fixing $a$. Then, by Lemma~\ref{prop:centraliser_fibre_preserving}, there exists a $y \in G_v$ such that $t^kg \chi_e = t^kgy$. But considering the relation  $a^{t^kg} = a^{\pm 1}$, this forces $y \in \langle a \rangle= G_e$, since $G_e$ is a maximal cyclic subgroup of a free group and hence equal to its own normaliser in $G_v$. 
		
		Thus $\chi_e$ restricts to an automorphism of $M_e$. While the choice of $h_e$ was not uniquely determined, any two choices will differ by an element of $\langle a \rangle$, so $\chi_M$ induces well defined outer automorphisms on each $M_e$. 
		
		Now we claim that whenever $M_e$ is free abelian of rank 2 -- when $a^{t^kg} = a$ -- then this forces $y=1$, and hence that $\chi_e$ is the identity map. This is is the same as saying that $w \chi_M = w^h$ for all $w \in M_e$. 
		
		To show this, recall that $\Phi^{t(v)} \mid_{G_v}$ is a finite order outer automorphism of $G_v$. This implies that there exists a $z \in G_v$ and an $0 \neq r \in \N$ such that $(t^kg)^r z$ is a central element of $M_v$ \cite[Proposition 4.1]{Levitt2015GBSrank}.
		
		In this case, we are assuming that $t^k g$ commutes with $a$ and hence $z$ must also commute with $a$, forcing $z \in \langle a \rangle = G_e$, again because $G_e$ is a maximal cyclic subgroup of a free group. 
		
		Hence, 
		\begin{itemize}
			\item $a \chi_e = a$, 
			\item $t^k g \chi_e = t^k g y$, $y \in \langle a \rangle$, 
			\item $(t^kg)^r z$ is a central element of $M_v$, where $0 \neq r \in \N$ and $z \in \langle a \rangle$. 
		\end{itemize}
		
		Moreover, since $\chi_M$ (and hence $\chi_e$) preserves $t$-exponent sums and the centre of $G_v$ is trivial, we also have that $\chi_M$ and $\chi_e$ both fix the central element  $(t^kg)^r z$. Putting this together, we get that: 
		$$
		(t^kg)^r z= \left\{ (t^kg)^r z \right\}\chi_e = (t^kg \chi_e)^r  z = (t^k g y)^r z = (t^k g)^r z y^r, 
		$$
		since $y,z \in \langle a \rangle$ and commute with $t^kg$. Hence $y=1$, as required.

		For the other case, we have that $M_e$ is a Klein Bottle group -- when $a^{t^kg} = a^{-1}$. Again we have that $ a \chi_e = a$ and $t^kg \chi_e = t^k g y$ for some $y \in \langle a \rangle$. If we further assume that $\chi_e$ is trivial as an outer automorphism, then this amounts to adding the restriction that $y \in \langle a^2 \rangle$. But then we get that $y = b^2$, where $b \in \langle a \rangle$ and we have that $a \chi_e = a^{b}$ and $t^kg \chi_e = (t^k g)^b$. Hence, by modifying $h_e$ (replacing it with $h_eb$) we deduce that  $w \chi_M = w^{h_e}$ for all $w \in M_e$.
		
		The last statement now follows immediately since the outer automorphism group of a Klein Bottle group is finite. \qedhere

\end{proof}
%
%

\begin{prop}
	\label{finiteindex}
	There is a finite index subgroup	
	$$\prod_{[v] \in V(T)/M} N_v  \leq \prod_{[v] \in V(T)/M} \MC(G_v;\{G_e\}_{\iota(e)=v};  \Phi^{t(v)} \mid_{G_v} ),$$ 
	such that for every element, $( X_v ) \in \prod_{[v] \in V(T)/M} N_v $, there exists a fibre and coset preserving automorphism $\chi_{\mathcal{M}}$ of $M$ such that: 
	
	\begin{enumerate}[(i)]
		\item $\chi_{\mathcal{M}} \in \Aut^T(M)$, 
		\item The restriction of $\chi_{\mathcal{M}}$ to $G$ lies in $\widehat{C}_0(\Phi) \leq \Aut^T(G)$, 
		\item For every (non-subdivision) vertex $v$ of $T$, the outer automorphism induced by $\chi_{\mathcal{M}}$ on $G_v$ is the same as that induced by $X_v$.
	\end{enumerate}

%
\end{prop}

\begin{proof}
	Since $T$ is an $M$-tree, there is a corresponding graph of groups $\mathcal{M}$ as constructed in Definition~\ref{def:quotient_gog}. In particular, we use the notation established there.

	

	
	
	
	The finite index subgroups $N_v$ are those provided by Proposition~\ref{prop:nice_mccool}, and their product is a finite index subgroup of $\prod_{[v] \in V(T)/M} \MC(G_v;\{G_e\}_{\iota(e)=v};  \Phi^{t(v)} \mid_{G_v} )$. 
	
	To start with, pick a representative $\chi_v$ of each outer automorphism $X_v$ 
	and apply Proposition~\ref{prop:nice_mccool} to produce an automorphism (which we will again call $\chi_v$)  of $M_v$ such that there exist elements, $h_e \in G_v$ with $w \chi_v = w^{h_e}$ for all $w \in M_e$. Note that each $\chi_v$ is a fibre and coset preserving automorphism of $M_v$ with respect to $G_v$. For the subdivision vertices, if any, let $\chi_v$ be the identity automorphism, and note that by definition this is fibre and coset preserving. For convenience in what follows, additionally set $h_e =1$ for the edge adjacent to $v$ (if the subdivision is required, there is exactly one).
	

	We can now define a graph of groups isomorphism of $\mathcal{M}$: the underlying graph map should be the identity, the edge group homomorphisms the identity, and the vertex group maps the automorphisms $\chi_v$ from above. 
	For an edge $e$ of $\mathcal{M}$ put $e^{+}=e \cdot g_e^{+}$, $e^{-}=e \cdot g_e^{-}$, $v= \tau(e) \cdot g_e^{+}$ and $u= \iota(e) \cdot g_e^{-}$. Finally, set $\delta_e = h_{e^{+}}$, so that the image of $e$ in the induced map on the path group $P(\mathcal{M})$ will be $ h_{e^{-}}^{-1}  e h_{e^{+}}$. 
	
	(Recall that the elements $g_e^{-}, g_e^{+}$ are defined in Definition~\ref{def:quotient_gog}. Namely, we describe the quotient graph by taking a representative of each orbit, and these group elements are defined so that $\iota(e) g_e^{-} , \tau(e) g_e^{+}$ equal their respective representatives. They are only unique up to the relevant vertex stabiliser.) 
	
	
	To see that this is a graph of groups isomorphism, it suffices to check the edge relations in the path group are satisfied.  Observe that if $x \in M_e$, then $\alpha_e(x) = x^{g_e^{+}}$ and $\alpha_{\overline{e}}(x) = x^{g_e^{-}}$. 
	
	Moreover, $\alpha_e(x)\chi_v = (x^{g_e^{+}})\chi_v =  x^{g_e^{+}h_{e^+}}$ and similarly $( x^{g_e^{-}})\chi_u =  x^{g_e^{-}h_{e^-}}$ for all $x \in M_e$. It follows that both sides of the relation $e \alpha_e(x) \overline{e} = \alpha_{\overline{e}}(x)$ are sent to $x^{g_e^-h_{e^{-}}}$ by our map.
	
	To summarise, we have produced a graph of groups automorphism, $\chi_{\mathcal{M}}$ with the following effect on generators of $P(\mathcal{M})$: 
	$$
	\begin{array}{rcll}
		m_v \chi_{\mathcal{M}} & = & m_v \chi_v, &\text{for all } m_v \in M_v \\
		e\chi_{\mathcal{M}} & = & h_{e^{-}}^{-1}  e h_{e^{+}} & \text{for all edges } e.
	\end{array}
	$$
	
	As this is a graph of groups automorphism, we get that $\chi_{\mathcal{M}} \in \Aut^T(M)$ by Proposition~\ref{prop:gog_auts_act}.

	We want to show that  $\chi_{\mathcal{M}}$ restricts to an automorphism of $G$ that commutes with $\Phi$ as an outer automorphism. So we invoke Lemma~\ref{prop:centraliser_fibre_preserving}. Hence, we need to show that  $\chi_{\mathcal{M}}$ is a fibre and coset preserving automorphism of $M$ with respect to $G$, as in Definition~\ref{def:coset_fibre_preserving}. Note that showing this is equivalent to showing that $m^{-1} (m \chi_{\mathcal{M}}) \in G$ for all $m \in M$.  We will do this in stages.

	For this, note that we have a map, $\rho: M \to \Z$ whose kernel is $G$ and such that, for any $v$, $\ker(\rho) \cap M_v=G_v$. We then invoke Proposition~\ref{extendhoms} to extend $\rho$ to a homomorphism, $\hat{\rho}: P(\mathcal{M}) \to \Z$. We work with the induced automorphism of the path group defined above. 

	To this end, let $K$ be the normal closure of the subgroups $G_v$ in the path group, $P(\mathcal{M})$. We claim that, for every generator, $x$, of $P(\mathcal{M})$ we have that $x^{-1} (x \chi_{\mathcal{M}}) \in K$. If $x \in M_v$, then $x^{-1} (x \chi_{\mathcal{M}}) = x^{-1} (x \chi_{v}) \in G_v$, since $\chi_v$ is a fibre and coset preserving automorphism of $M_v$ with respect to $G_v$. If $x=e$, an edge, then  $e^{-1} (e \chi_{\mathcal{M}}) = e^{-1} h_{e^{-}}^{-1}  e h_{e^{+}} \in K$, since $h_{e^{-}} \in G_u$ and $h_{e^{+}} \in G_v$ for some vertices $u$ and $v$. 
	
	But since $x^{-1} (x \chi_{\mathcal{M}}) \in K$ for a generating set, it follows using normality of $K$ that this holds for all elements of $P(\mathcal{M})$. 
	But now, if $x \in M$, then  $x^{-1} (x \chi_{\mathcal{M}}) \in K \cap M$. Since the kernel of each $\rho_v:= \rho\mid_{M_v}$ is $G_v$, it follows that $K \leq \ker(\hat{\rho})$ and so $K \cap M \leq  \ker(\hat{\rho}) \cap M = \ker(\rho)=G$, since $\hat{\rho}$ extends $\rho$. 
	
	Thus we have shown that $x^{-1} (x \chi_{\mathcal{M}})  \in G$ for all $x \in M$. This shows that $\chi_{\mathcal{M}}$ is a fibre and coset preserving automorphism of $M$ with respect to $G$. It immediately follows that the restriction of $\chi_{\mathcal{M}}$ to $G$ lies in $\Aut^T(G)$. In fact, by Lemma~\ref{prop:centraliser_fibre_preserving}, the restriction of $\chi_{\mathcal{M}}$ to $G$ lies in $\widehat{C}_0(\Phi)$. 
	
	Moreover, by construction, for every vertex $v$ of $T$, the outer automorphism induced by $\chi_{\mathcal{M}}$ on $G_v$ is the same as that induced by $\chi_v$. \qedhere

\end{proof}

We are now in a position to prove our main theorem.
\centralisers*

\begin{proof}
We define $C_0(\Phi)$ and $\mu$ as in Notation~\ref{mudefined}. The kernel of $\mu$ is a finitely generated free abelian group, by Theorem~\ref{bassjiang}. The image of $\mu$ is isomorphic to a finite index subgroup $\prod_{[v] \in V(T)/M} \MC(G_v;\{G_e\}_{\iota(e)=v};  \Phi^{t(v)} \mid_{G_v} )$, since by Proposition~\ref{projection} $\mu(C_0(\Phi)))$ is isomorphic via $\rho$ to some subgroup of this product, and by Proposition~\ref{finiteindex} $\rho(\mu(C_0(\Phi)))$ contains its finite index subgroup $N$. The last statement, that each equivariant McCool group has type VF, is given by Theorem~\ref{mccool}. \qedhere

\end{proof}

					%
					%
					%
					%
					%
					%

\bibliographystyle{plain}

\end{document}